\documentclass[11pt]{article}
\usepackage{amsmath}
\usepackage{dsfont}
\usepackage{mathrsfs}
\usepackage{amsmath,amssymb}
\usepackage{amsfonts}
\usepackage{hyperref}
\usepackage{amsthm}
\usepackage{graphicx}
\usepackage{subfigure}
\usepackage{xcolor}

\usepackage{bbm}
\usepackage{mathrsfs}
\usepackage{txfonts}
\usepackage{color}
\usepackage{pict2e}
\usepackage{palatino,epsfig,latexsym}
\usepackage{epsf,xy,epic,amscd}
\usepackage{lineno}

\usepackage{cite}

\theoremstyle{plain}
\newtheorem{theorem}{Theorem}[section]
\newtheorem{lemma}[theorem]{Lemma}

\newtheorem{proposition}[theorem]{Proposition}

\newtheorem{Bounded Diameter Lemma}[theorem]{Bounded Diameter Lemma}
\theoremstyle{definition}

\newtheorem{remark}[theorem]{Remark}
\newtheorem{acknowledgements}[theorem]{Acknowledgements}

\newcommand{\Hmm}[1]{\leavevmode{\marginpar{\tiny%
$\hbox to 0mm{\hspace*{-0.5mm}$\leftarrow$\hss}%
\vcenter{\vrule depth 0.1mm height 0.1mm width \the\marginparwidth}%
\hbox to
0mm{\hss$\rightarrow$\hspace*{-0.5mm}}$\\\relax\raggedright #1}}}

\hfuzz=\maxdimen
\DeclareFixedFont{\Acknowledgment}{OT1}{cmr}{bx}{n}{14pt}
\begin{document}

\title{Combinatorial Yamabe flow on hyperbolic bordered surfaces}
\author{Shengyu Li, Xu Xu, Ze Zhou}
\date{}

\maketitle

\begin{abstract}
This paper studies the combinatorial Yamabe flow on hyperbolic bordered surfaces.
We show that the flow exists for all time and converges exponentially fast to conformal factor which produces a hyperbolic surface whose lengths of boundary components are equal to prescribed positive numbers. This provides an algorithm to such problems.


\end{abstract}



\section{Introduction}
Discrete conformal structure on polyhedral surfaces is a discrete analogue of the conformal structure on smooth surfaces. Over the past decades, the discrete conformal structures on closed surfaces have been extensively studied. See e.g.
\cite{BPS,ChowLuo-jdg,GT,GGLSW,GLSW,GLW,GL,Guo,Luo1,LW,L3,SWGL,Wu,WGS,WZ,X1,X2,T1,ZGZLYG}.
In order to investigate the piecewise linear conformal geometry, Luo~\cite{Luo1} introduced the combinatorial Yamabe flow for piecewise flat metrics on triangulated closed surfaces.
An analogous flow was studied by Guo \cite{Guo} on surface with boundary in hyperbolic geometry background.
He proved that along the combinatorial Yamabe flow, any initial hyperbolic bordered surface converges to a complete hyperbolic surface with cusps.
Motivated by the works of Luo \cite{Luo1} and Guo \cite{Guo}, in this paper we introduce a more general combinatorial Yamabe flow on ideally triangulated bordered surfaces, aiming at finding hyperbolic metrics on bordered surfaces with prescribed lengths of geodesic boundaries.
We prove the longtime existence and global convergence of this flow.

Let $T^{\ast}$ be a triangulation of a closed surface $S$ with the sets of vertices, edges and faces denoted by $T^{0},T^{1},T^{2}$, respectively.
Suppose $N(T^{0})$ is a small open regular disjoint neighborhood of the union of all vertices $T^{0}$.
Then $\Sigma=S \setminus N(T^{0})$ is a compact surface with $|T^{0}|$ boundary components.
The intersection $T=T^{\ast}\cap\Sigma$ is called an ideal triangulation of the surface $\Sigma$.
The intersections $E=T^{1}\cap\Sigma$, $F=T^{2}\cap\Sigma$ are called ideal edges and ideal faces of $\Sigma$ respectively.
The intersection of an ideal face and $\partial\Sigma$ are called B-arcs.
For simplicity, we write the set of boundary components as $\{1,2,\cdots,n\}$. Let $ij$ denote the ideal edge between two adjacent boundary components $i,j$. Similarly, we use $ijk$ to represent the ideal face adjacent to boundary components $i,j,k$. See the work of Luo \cite{Luo2} for basic background on ideal triangulation.

Let $E, F$ be the sets of ideal edges and ideal faces of $T$, respectively.
A discrete hyperbolic metric associated to $T$ is a vector $l\in\mathbb{R}_{+}^{\mid E\mid}$ assigning each ideal edge $ij$ a positive number $l_{ij}$.
For an ideal face adjacent to boundary components $i,j,k$, by Lemma \ref{21} in next section, there exists an essentially unique right-angled geodesic hexagon whose three non-pairwise adjacent edges having lengths $l_{ij},l_{jk},l_{ki}$.
Gluing all such geodesic hexagons produces a hyperbolic surface with geodesic boundary.

A discrete conformal factor is a vector $w\in\mathbb{R}^n$ assigning each boundary component $i$ a number $w_{i} $. Fix a hyperbolic metric $l^{0}\in\mathbb{R}_{+}^{\mid E\mid}$ on $(\Sigma,T)$. Following Guo \cite{Guo}, we say a discrete conformal factor is admissible if
\[
w_{i}+w_{j}>-\ln\cosh\frac{l_{ij}^{0}}{2}
\]
for each edges $ij\in E$. Given a metric $l^{0}$ and an admissible discrete conformal factor $w$, we get a new metric $l=w\ast l^{0} \in\mathbb{R}_{+}^{\mid E\mid}$ satisfying
\begin{equation}
\nonumber
\cosh\frac{l_{ij}}{2}=e^{w_{i}+w_{j}}\cosh\frac{l^{0}_{ij}}{2}.
\end{equation}
Suppose $\emph{W}\subset\mathbb{R}^{n}$ is the space of all admissible discrete conformal factors. For an ideal edge $e=ij\in E$, we define
\begin{equation}
\nonumber
W^{e}=\left\{w\in\mathbb{R}^n:
w_{i}+w_{j}>-\ln\cosh\frac{l_{ij}^{0}}{2}\right\}.
\end{equation}
Similarly, for an ideal face $f=ijk$, write
\begin{equation}
\nonumber
W^{f}=\left\{w\in\mathbb{R}^n:
w_{r}+w_{s}>-\ln\cosh\frac{l_{rs}^{0}}{2}, \{r,s\}\subset\{i,j,k\}\right\}.
\end{equation}
Apparently, $W^{e}$ and $W^{f}$ are convex sets. Thus
\[
W=\cap_{e\in E} W^{e}=\cap_{f\in F} W^{f}
\]
is also a convex set. Note that each $w\in W$ gives a metric $l=w\ast l^{0}\in\mathbb{R}_{+}^{\mid E\mid}$ on a hyperbolic surface $\Sigma_{l}$ with geodesic boundary.
For $i=1,2,\cdots,n$, let $B_{i}$ be the length of geodesic boundary component $i$ of surface $\Sigma_{l}$. We have
\[
B_{i}=\sum_{f\in F_{i}}\theta_{i}^{f},
\]
where $\theta^{f}_{i}$ is the length of the B-arc marked by boundary component $i$ and $F_{i}$ is the set of ideal faces having boundary component $i$.
This gives rise to a map
\[
\begin{aligned}
\psi: \qquad\qquad W \qquad \quad &\longrightarrow \qquad\quad\mathbb{R}_{+}^n\qquad\qquad\\
\big(w_{1},w_{2},\cdots,w_n\big)&\longmapsto \big(B_{1},B_{2},\cdots,B_n\big).\\
\end{aligned}
\]

By variational principle, Guo \cite{Guo} proved the following result.

\begin{theorem}[Guo]\label{1234}
The map $\psi$ is a diffeomorphism.
\end{theorem}

What is more, Guo \cite{Guo} investigated the following combinatorial Yamabe flow
\begin{equation}
\label{202159}
\frac{dw_{i}(t)}{dt}=B_{i}
\end{equation}
with an initial vector $w(0)\in W$.
He proved that along the flow (\ref{202159}) any initial hyperbolic bordered surface converges to a complete hyperbolic surface with cusps.

Unfortunately, there is a gap in Guo's proof \cite{Guo} of Theorem \ref{1234} (see Remark \ref{20212021} for details). In this paper we will fix the gap and provide a complete proof. Meanwhile, to construct more general bordered hyperbolic surfaces, we consider the flow
\begin{equation}
\label{1111111111}
\frac{dw_{i}(t)}{dt}=B_{i}-b_{i},
\end{equation}
where each $b_{i}$ is a prescribed positive number. By using Theorem \ref{1234} and introducing suitable Lyapunov functions, we prove the following result.
\begin{theorem}
\label{123456}
The flow (\ref{1111111111}) exists for all time and converges exponentially fast to an admissible discrete conformal factor which produces a hyperbolic surface with geodesic boundary components having lengths $b_{1},b_{2},\cdots,b_{n}$, respectively.
\end{theorem}

The paper is organized as follows. In Section \ref{S1}, we introduce some basic properties of hexagons.
In Section \ref{2021.10.4}, we give a complete proof of Theorem \ref{1234}.
In Section \ref{2021123}, we derive Theorem \ref{123456}.

\section{Preliminaries}\label{S1}\noindent
We establish several properties on right-angled hyperbolic hexagons. The lemma below is a classical result. We refer to Buser's book \cite[Lemma\,1.7.1]{buser} for a proof.
\begin{lemma}\label{21}
For any three positive real numbers, there exists a right-angled geodesic hexagon in the hyperbolic plane unique up to isometry such that the lengths of the three non-pairwise adjacent edges of the hexagon are the three numbers.
\end{lemma}

Let $l_{ij},l_{jk},l_{ki}$ be the lengths of three non-pairwise adjacent edges $ij,jk,ki$ of a right-angled hyperbolic hexagon and let $\theta_{i},\theta_{j},\theta_{k}$ be the lengths of edges opposite to $jk,ki,ij$, respectively. The cosine law for right-angled hyperbolic hexagons gives
\begin{equation}\label{2026}
\cosh\theta_{i}=\frac{\cosh l_{jk}+\cosh l_{ki}\cosh l_{ij}}{\sinh l_{ki}\sinh l_{ij}}.
\end{equation}
Similarly, the sine law for right-angled hyperbolic hexagons shows
\begin{equation}\label{2027}
\frac{\sinh\theta_{i}}{\sinh l_{jk}}=\frac{\sinh\theta_{j}}{\sinh l_{ki}}=\frac{\sinh\theta_{k}}{\sinh l_{ij}}.
\end{equation}
See also Buser's book \cite[Chap.\,2]{buser} for details.

Fix a right-angled hyperbolic hexagon with lengths $l_{ij}^{0},l_{jk}^{0},l_{ki}^{0}$ of three non-pairwise edges. Set
\begin{equation}
\nonumber
W^{\diamond}=\left\{(w_{i},w_{j},w_{k})\in\mathbb{R}^3:
w_{r}+w_{s}>-\ln\cosh\frac{l_{rs}^{0}}{2}, \{r,s\}\subset\{i,j,k\}\right\}.
\end{equation}
For every vector $(w_{i},w_{j},w_{k})\in W^{\diamond}$,
we get three lengths $l_{ij},l_{jk},l_{ki}$  satisfying
\begin{equation}
\label{123}
\cosh\frac{l_{ij}}{2}=e^{w_{i}+w_{j}}\cosh\frac{l^{0}_{ij}}{2}.
\end{equation}

In this way, we regard $\theta_{i},\theta_{j},\theta_{k}$ as smooth functions of $(w_{i},w_{j},w_{k})\in W^{\diamond}$.
The next result was due to Guo \cite{Guo}.
To make the paper be more self-contained, here we give an independent proof via diagonally dominant property. We mention that the lemma below is slightly stronger than Guo's original  result \cite{Guo}.
\begin{lemma}\label{L-1-1}
The Jacobian matrix of $\theta_{i},\theta_{j},\theta_{k}$ in terms of $w_{i},w_{j},w_{k}$ is symmetric, diagonally dominant and negative definite.
\end{lemma}
\begin{proof}
We begin by showing the Jacobian matrix is symmetric. Differentiating \eqref{2026} with respect to $l_{jk}$ gives
\begin{equation}\nonumber
\frac{\partial\theta_{i}}{\partial l_{jk}}=\frac{\sinh l_{jk}}{\sinh\theta_{i}\sinh l_{ki}\sinh l_{ij}}.
\end{equation}
What is more, we have
\begin{equation}\nonumber
\frac{\partial\theta_{i}}{\partial l_{ki}}=-\frac{\cosh l_{ij}+\cosh l_{jk}\cosh l_{ki}}{\sinh\theta_{i}\sinh l_{ij}\sinh^{2} l_{ki}}
\end{equation}
and
\begin{equation}\nonumber
\frac{\partial\theta_{i}}{\partial l_{ij}}=-\frac{\cosh l_{ki}+\cosh l_{jk}\cosh l_{ij}}{\sinh\theta_{i}\sinh l_{ki}\sinh^{2} l_{ij}}.
\end{equation}
Meanwhile, it follows from \eqref{123} that
\[
\frac{\partial l_{ij}}{\partial w_{j}}=\frac{2\sinh l_{ij}}{\cosh l_{ij}-1}.
\]
In addition,
\[
\frac{\partial l_{ki}}{\partial w_{j}}=0,\quad\frac{\partial l_{jk}}{\partial w_{j}}=\frac{2\sinh l_{jk}}{\cosh l_{jk}-1}.
\]
In light of the above formulas, we derive
\begin{equation}
\begin{aligned}
\nonumber
\frac{\partial\theta_{i}}{\partial w_{j}}&=\frac{\partial\theta_{i}}{\partial l_{jk}}\frac{\partial l_{jk}}{\partial w_{j}}+\frac{\partial\theta_{i}}{\partial l_{ki}}\frac{\partial l_{ki}}{\partial w_{j}}+\frac{\partial\theta_{i}}{\partial l_{ij}}\frac{\partial l_{ij}}{\partial w_{j}}\\
&=-\frac{2(\cosh l_{jk}+\cosh l_{ki}-\cosh l_{ij}+1)}{\sinh\theta_{i}\sinh l_{ki}\sinh l_{ij}(\cosh l_{ij}-1)}.
\end{aligned}
\end{equation}
Similarly,
\begin{equation}
\nonumber
\frac{\partial\theta_{j}}{\partial w_{i}}=-\frac{2(\cosh l_{jk}+\cosh l_{ki}-\cosh l_{ij}+1)}{\sinh\theta_{j}\sinh l_{jk}\sinh l_{ij}(\cosh l_{ij}-1)}.
\end{equation}
Applying \eqref{2027}, we easily obtain
\[
\sinh\theta_{j}\sinh l_{jk}\sinh l_{ij}=\sinh\theta_{i}\sinh l_{ki}\sinh l_{ij}:=A.
\]
As a result,
\begin{equation}\nonumber
\frac{\partial\theta_{i}}{\partial w_{j}}=\frac{\partial\theta_{j}}{\partial w_{i}}=-\frac{2}{A}\frac{\cosh l_{jk}+\cosh l_{ki}-\cosh l_{ij}+1}{\cosh l_{ij}-1}.
\end{equation}
That means the Jacobian matrix is symmetric.

Next we check that the Jacobian matrix is diagonally dominant. A routine computation gives
\[
\frac{\partial\theta_{i}}{\partial w_{i}}=\frac{-2}{A}\left(\frac{\cosh l_{ij}+\cosh l_{jk}\cosh l_{ki}}{\cosh l_{ki}-1}+\frac{\cosh l_{ki}+\cosh l_{jk}\cosh l_{ij}}{\cosh l_{ij}-1}\right)<0
\]
and
\[
\frac{\partial\theta_{i}}{\partial w_{k}}=-\frac{2}{A}\frac{\cosh l_{jk}+\cosh l_{ij}-\cosh l_{ki}+1}{\cosh l_{ki}-1}.
\]
Now we divide the proof into the following three cases:
\begin{itemize}
\item[(i)]$\cosh l_{ki}>\cosh l_{jk}+\cosh l_{ij}+1.$ Then
\begin{equation}
\begin{aligned}
\nonumber
\left|\frac{\partial\theta_{i}}{\partial w_{i}}\right|-\left|\frac{\partial\theta_{i}}{\partial w_{j}}\right|-\left|\frac{\partial\theta_{i}}{\partial w_{k}}\right|
&=-\frac{\partial\theta_{i}}{\partial w_{i}}+\frac{\partial\theta_{i}}{\partial w_{j}}-\frac{\partial\theta_{i}}{\partial w_{k}}\\
&=\frac{4}{A}\frac{\cosh l_{jk}\cosh l_{ki}+\cosh l_{ij}}{\cosh l_{ki}-1}\\
&>0.\\
\end{aligned}
\end{equation}
\item[(ii)] $\cosh l_{ij}>\cosh l_{jk}+\cosh l_{ki}+1.$\
Similar arguments to the first case yield
\[
\left|\frac{\partial\theta_{i}}{\partial w_{i}}\right|-\left|\frac{\partial\theta_{i}}{\partial w_{j}}\right|-\left|\frac{\partial\theta_{i}}{\partial w_{k}}\right|>0.
\]
\item[(iii)] $\cosh l_{ki}\leq\cosh l_{jk}+\cosh l_{ij}+1$ and $\cosh l_{ij}\leq\cosh l_{ki}+\cosh l_{jk}+1.$ We obtain
\begin{equation}
\begin{aligned}
\nonumber
\left|\frac{\partial\theta_{i}}{\partial w_{i}}\right|-\left|\frac{\partial\theta_{i}}{\partial w_{j}}\right|-\left|\frac{\partial\theta_{i}}{\partial w_{k}}\right|
&=-\frac{\partial\theta_{i}}{\partial w_{i}}+\frac{\partial\theta_{i}}{\partial w_{j}}+\frac{\partial\theta_{i}}{\partial w_{k}}\\
&=\frac{4(\cosh l_{jk}+1)}{A}\\
&>0.\\
\end{aligned}
\end{equation}
\end{itemize}
In summary, we always have
\[
\left|\frac{\partial\theta_{i}}{\partial w_{i}}\right|-\left|\frac{\partial\theta_{i}}{\partial w_{j}}\right|-\left|\frac{\partial\theta_{i}}{\partial w_{k}}\right|>0,
\]
which implies the Jacobian matrix is diagonally dominant.
Since $\frac{\partial\theta_{i}}{\partial w_{i}}<0$, it is easy to see the Jacobian matrix is negative definite. We thus finish the proof.
\end{proof}

To prove Theorem \ref{1234}, we need to describe how $\theta_{i}, \theta_{j}, \theta_{k}$ behave as $(w_{i},w_{j},w_{k})$ tends to the boundary of $W^{\diamond}$.
The following lemma was also obtained by Guo \cite{Guo}. As before, we include the proof for the sake of completeness.

\begin{lemma}\label{2.4}
Let $\theta_{i}, \theta_{j}, \theta_{k}$ be as above. Then
\begin{subequations}
\begin{align}
\lim_{(w_{i},w_{j},w_{k}) \to (+\infty,c_{1},c_{2})}&\theta_{i}=0, \label{MMa}\\
\lim_{(w_{i},w_{j},w_{k}) \to (+\infty,+\infty,c_{3})}&\theta_{i}=0,\label{MMb}\\
\lim_{(w_{i},w_{j},w_{k}) \to (+\infty,+\infty,+\infty)}&\theta_{i}=0,\label{MMc}\\
\lim_{(w_{i},w_{j},w_{k}) \to (c_{4},c_{5},c_{6})}&\theta_{i}=+\infty,\label{MMd}
\end{align}
\end{subequations}
where $c_{1},c_{2},c_{3},c_{4},c_{5},c_{6}$ are constants satisfying
\[
c_{1}+c_{2}\geq-\ln\cosh\frac{l_{jk}^{0}}{2}
\]
and
\[
c_{4}+c_{5}=-\ln\cosh\frac{l_{ij}^{0}}{2},\quad c_{5}+c_{6}\geq-\ln\cosh\frac{l_{jk}^{0}}{2},\quad c_{4}+c_{6}\geq-\ln\cosh\frac{l_{ki}^{0}}{2}.
\]
\end{lemma}

\begin{proof}
By the cosine law of right-angled hyperbolic hexagons, we write
\begin{equation}
\nonumber
\cosh\theta_{i}=X+Y,
\end{equation}
where
\[
X=\frac{\cosh l_{jk}}{\sinh l_{ki}\sinh l_{ij}},\quad Y=\frac{\cosh l_{ki}\cosh l_{ij}}{\sinh l_{ki}\sinh l_{ij}}.
\]
To simplify notations, set
\begin{equation}\nonumber
\rho_{jk}=2\cosh^{2}\frac{l_{jk}^{0}}{2}, \quad \rho_{ki}=2\cosh^{2}\frac{l_{ki}^{0}}{2}, \quad \rho_{ij}=2\cosh^{2}\frac{l_{ij}^{0}}{2}.
\end{equation}
Therefore, we rewrite formula \eqref{123} as
\begin{equation}
\label{717777777777778}
\cosh l_{jk}=\rho_{jk}e^{2w_{j}+2w_{k}}-1.
\end{equation}
Similarly,
\begin{equation}
\label{71777777777777}
\cosh l_{ki}=\rho_{ki}e^{2w_{i}+2w_{k}}-1,\quad\cosh l_{ij}=\rho_{ij}e^{2w_{i}+2w_{j}}-1.
\end{equation}
As $(w_{i},w_{j},w_{k}) \to (+\infty,c_{1},c_{2})$, we have
\[
\cosh l_{jk}\to\rho_{jk}e^{2c_{1}+2c_{2}}-1,\quad l_{ki}\to +\infty,\quad l_{ij}\to +\infty.
\]
It is easy to see
\[
X=\frac{\cosh l_{jk}}{\sinh l_{ki}\sinh l_{ij}}\to0,\, \quad Y=\frac{\cosh l_{ki}\cosh l_{ij}}{\sinh l_{ki}\sinh l_{ij}}\to1.
\]
Consequently, $\cosh\theta_{i}\to1$, which implies (\ref{MMa}).

Next we prove (\ref{MMb}) and (\ref{MMc}) simultaneously. Combining (\ref{717777777777778}) and (\ref{71777777777777}) gives
\begin{equation}
\begin{aligned}
\label{098765}
X&=\frac{\rho_{jk}e^{2w_{j}+2w_{k}}-1}{\sqrt{\left[(\rho_{ki}e^{2w_{k}+2w_{i}}-1)^{^{2}}-1\right] \big[(\rho_{ij}e^{2w_{i}+2w_{j}}-1)^{^{2}}-1\big]}}\\
&=\frac{\rho_{jk}e^{2w_{j}+2w_{k}}-1}{e^{2w_{j}+2w_{k}}}\frac{1}{\sqrt{\rho_{ki}\rho_{ij}}}
\frac{1}{\sqrt{\rho_{ki}e^{4w_{i}}-2e^{2w_{i}-2w_{k}}}}
\frac{1}{\sqrt{\rho_{ij}e^{4w_{i}}-2e^{2w_{i}-2w_{j}}}}.
\end{aligned}
\end{equation}
As $(w_{i},w_{j},w_{k}) \to (+\infty,+\infty ,c_{2})$ or $(+\infty,+\infty,+\infty)$,
we have
\[
\frac{\rho_{jk}e^{2w_{j}+2w_{k}}-1}{e^{2w_{j}+2w_{k}}}\to\rho_{jk}.
\]
In addition,
\[
\frac{1}{\sqrt{{\rho_{ki}e^{4w_{i}}-2e^{2w_{i}-2w_{k}}}}}=\frac{1}{\sqrt{e^{2w_{i}}(\rho_{ki}e^{2w_{i}}-2e^{-2w_{k}}})}\to 0
\]
and
\[
\frac{1}{\sqrt{\rho_{ij}e^{4w_{i}}-2e^{2w_{i}-2w_{j}}}}\to 0.
\]
As a result, $X\to0$. Meanwhile, by (\ref{717777777777778}) and (\ref{71777777777777}), we obtain
\[
l_{ki}\to +\infty,\quad l_{ij}\to +\infty,
\]
which yields $Y\to1$. It follows that $\cosh\theta_{i}\to1$. Thus we derive (\ref{MMb}) and (\ref{MMc}).

It remains to prove (\ref{MMd}). Since $(w_{i},w_{j},w_{k}) \to (c_{4},c_{5},c_{6})$ with
\[
c_{4}+c_{5}=-\ln\cosh\frac{l_{ij}^{0}}{2},
\]
we immediately get
\[
l_{ij}\to 0.
\]
Meanwhile,
\[
\sinh l_{ki}\to \sqrt{\rho_{ki}^{2}e^{4c_{4}+4c_{6}}-2\rho_{ki}e^{2c_{4}+2c_{6}}}.
\]
Hence
\[
X\geq\frac{1}{\sinh l_{ki}\sinh l_{ij}}\to+\infty.
\]
Noting that $Y\geq0$, we have $\cosh\theta_{i}=X+Y\to+\infty$, which yields $\theta_{i}\to+\infty$.
\end{proof}

The above lemma does not list all cases that $(w_{i},w_{j},w_{k})$ tends to the boundary of $W^{\diamond}$. In fact, this is the main obstruction why Guo \cite{Guo} did not give a complete proof of Theorem \ref{1234}.
Fortunately, we have the following lemma which complements the ignored case.
\begin{lemma}\label{2.3}Given $\theta_{i}, \theta_{j}, \theta_{k}$ as above, we have
\begin{equation}\label{5}
\lim_{(w_{i},w_{j},w_{k}) \to (-\infty,+\infty,+\infty)}\theta_{i}=+\infty.
\end{equation}
\end{lemma}
\begin{proof}
As before, it suffices to show $X\to+\infty$ as $(w_{i},w_{j},w_{k}) \to (-\infty,+\infty,+\infty)$. In this case,  one is ready to see
\[
\frac{\rho_{jk}e^{2w_{j}+2w_{k}}-1}{e^{2w_{j}+2w_{k}}}\to\rho_{jk}.
\]
Moreover,
\[
\frac{1}{\sqrt{\rho_{ki}e^{4w_{i}}-2e^{2w_{i}-2w_{k}}}}\to+\infty,\quad\frac{1}{\sqrt{\rho_{ij}e^{4w_{i}}-2e^{2w_{i}-2w_{j}}}}\to+\infty.
\]
In view of (\ref{098765}), we get $X\to+\infty$, which concludes the lemma.
\end{proof}

\section{Proof of Theorem 1.1} \label{2021.10.4}\noindent
Before giving the proof of Theorem \ref{1234}, we establish some preparatory results.
\begin{lemma}\label{L-1-3}
The Jacobian matrix $L$ of $(B_{1},B_{2},\cdots,B_{n})$ in terms of $(w_{1},w_{2},\cdots,w_{n})$ is symmetric, diagonally dominant and negative definite.
\end{lemma}
\begin{proof}
Suppose $i\neq j$. If there exists no ideal edge between the boundary components $i,j$, it is easy to see
\[
\frac{\partial B_{i}}{\partial w_{j}}=\frac{\partial B_{j}}{\partial w_{i}}=0.
\]
Otherwise, there exists an ideal edge between $i,j$.  Let $f_{1}, f_{2}$ be the two ideal faces adjacent to the edge $ij$.
By Lemma \ref{L-1-1}, a simple computation gives
\[
\frac{\partial B_{i}}{\partial w_{j}}=\frac{\partial\theta_{i}^{f_{1}}}{\partial w_{j}}+\frac{\partial\theta_{i}^{f_{2}}}{\partial w_{j}}=\frac{\partial\theta_{j}^{f_{1}}}{\partial w_{i}}+\frac{\partial\theta_{j}^{f_{2}}}{\partial w_{i}}=\frac{\partial B_{j}}{\partial w_{i}}.
\]
Thus the Jacobian matrix $L$ is symmetric. Moreover, notice that
\begin{equation}
\begin{aligned}
\nonumber
\left|\frac{\partial B_{i}}{\partial w_{i}}\right|-\sum_{j\neq i}^{n}\left|\frac{\partial B_{i}}{\partial w_{j}}\right|
&=\left|\frac{\partial B_{i}}{\partial w_{i}}\right|-\sum_{j\neq i}^{n}\left|\frac{\partial B_{j}}{\partial w_{i}}\right|\\
&\geq\sum_{f\in F_{i}}\left(\left|\frac{\partial\theta_{i}^{f}}{\partial w_{i}}\right|-\left|\frac{\partial\theta_{j}^{f}}{\partial w_{i}}\right|-\left|\frac{\partial\theta_{k}^{f}}{\partial w_{i}}\right|\right).\\
\end{aligned}
\end{equation}
To show $L$ is diagonally dominant, we need to check
\[
\left|\frac{\partial\theta_{i}^{f}}{\partial w_{i}}\right|-\left|\frac{\partial\theta_{j}^{f}}{\partial w_{i}}\right|-\left|\frac{\partial\theta_{k}^{f}}{\partial w_{i}}\right|>0,
\]
which is also asserted by Lemma \ref{L-1-1}. Finally, taking into consideration that
\[
\frac{\partial B_{i}}{\partial w_{i}}=\sum_{f\in F_{i}}\frac{\partial\theta_{i}^{f}}{\partial w_{i}}<0,
\]
we easily prove $L$ is negative definite. 
\end{proof}

Let us consider the 1-form $\alpha=\sum_{i=1}^nB_{i}dw_{i}.$ A simple computation yields
\[
d\alpha=\sum_{i=1}^n\sum_{j=1}^n\Big(\frac{\partial B_{i}}{\partial w_{j}}-\frac{\partial B_{j}}{\partial w_{i}}\Big)dw_{i}\wedge dw_{j}=0.
\]
Thus $\alpha$ is closed. Recall that $W$ is a convex set. By Poincar\'{e}'s Lemma, for any $w(0) \in W$, the following function
\begin{equation} \nonumber
\Phi(w)=-\int_{w(0)}^{w} \sum_{i=1}^n B_{i}dw_{i},
\end{equation}
is well-defined. Moreover, we have the following property. 

\begin{lemma}
$\Phi(w)$ is a strictly convex function in $W$.
\end{lemma}
\begin{proof}
Note that the Hessian of $\Phi(w)$ is equal to $-L$, which is positive definite by Lemma \ref{L-1-3}. Therefore, $\Phi(w)$ is strictly convex.
\end{proof}

\begin{proposition}
The map $\psi$ is injective.
\end{proposition}
\begin{proof}\label{20220227}
Since $\psi$ is the gradient map of the strictly convex $\Phi(w)$, the statement follows from the following Lemma \ref{202158}.
\end{proof}

The lemma below is a standard result in analysis. See e.g. \cite{GHZ,GHZ2} for a proof.
\begin{lemma}
\label{202158}
Suppose $\Omega\subset\mathbb{R}^{n}$ is an open convex set and $h : \Omega\to\mathbb{R}$ is a strictly convex smooth function. Then the gradient map $\nabla h: \Omega\to\mathbb{R}^{n}$ is injective.
\end{lemma}

We are now ready to prove Theorem \ref{1234} based on the main ideas of Guo's original proof \cite{Guo}. As mentioned before, there is a gap in Guo's proof regarding to properness of the map $\psi$ (see Remark \ref{20212021}). Here we fix the gap and reprove the other parts to make the overall proof clearer.

\medskip
\begin{proof}[Proof of Theorem \ref{1234}]
We have the following claims:
\begin{itemize}
\item[$(i)$] $\psi$ is smooth and $\psi(W)\subset\mathbb{R}_{+}^{n}$. It is straightforward.
\item[$(ii)$] $\psi$ is injective. This has been proved in Proposition 3.3.
\item[$(iii)$] $\psi: W\to\mathbb{R}_{+}^{n}$ is proper.
It suffices to verify that some $B_{i}$ became infinity or zero as $w$ tends to the boundary of $W$.
Let $\{w^{(m)}\}\subset W$ be a sequence approaching the boundary. We divide the situation into the following cases:
\begin{itemize}
\item[$(a)$] $w^{(m)}$ approaches the $-\infty$-boundary of $W$.
Precisely, there exists $i\in \{1,2,\cdots,n\}$ satisfying $w_{i}^{(m)}\to-\infty$. Since $\{w^{(m)}\}\subset W$, for every boundary component $r$ adjacent to $i$, we have
\[
w_{i}^{(m)}+w_{r}^{(m)}>-\ln\cosh\frac{l_{ir}^{0}}{2}.
\]
Hence $w_{r}^{(m)}\to+\infty$.
Due to Lemma \ref{2.3}, for each $f\in F_{i}$ we have $(\theta_{i}^{f})^{(m)}\to+\infty$, which implies $B_{i}^{(m)}\to+\infty$.
\item[$(b)$] $w^{(m)}$ approaches the $+\infty$-boundary of $W$. Namely, there exists $i\in \{1,2,\cdots,n\}$ such that $w_{i}^{(m)}\to+\infty$.
For each $f\in F_{i}$, it follows from formulas \eqref{MMa}-\eqref{MMc} in Lemma \ref{2.4} that $(\theta_{i}^{f})^{(m)}\to0$. Thus
\[
B_{i}^{(m)}=\sum_{f\in F_{i}}(\theta_{i}^{f})^{(m)}\to0.
\]
\item[$(c)$] $w^{(m)}$ approaches the finite boundary of $W$. In other words, there exists an ideal face $f=ijk$ subject to
\[
\left(w_{i}^{(m)},w_{j}^{(m)},w_{k}^{(m)}\right) \to (c_{4},c_{5},c_{6}),
\]
where $c_{4},c_{5},c_{6}$ satisfy
\[
c_{4}+c_{5}=-\ln\cosh\frac{l_{ij}^{0}}{2}, \,c_{5}+c_{6}\geq-\ln\cosh\frac{l_{jk}^{0}}{2},\,c_{4}+c_{6}\geq-\ln\cosh\frac{l_{ki}^{0}}{2}.
\]
By formula \eqref{MMd}, we get $(\theta_{i}^{f})^{(m)}\to+\infty$. Consequently, $B_{i}^{(m)}\to+\infty$.
\end{itemize}
\end{itemize}
By Brouwer's theorem on invariance of domain, the former two claims imply that $\psi(W)$ is a non-empty open set in $\mathbb{R}_{+}^{n}$. Combining with the third claim, $\psi(W)$ is both open and closed in $\mathbb{R}_{+}^{n}$. Since $\mathbb{R}_{+}^{n}$ is connected, $\psi(W)=\mathbb{R}_{+}^{n}$.
\end{proof}

\begin{remark}\label{20212021}
In proving claim $(iii)$, Guo \cite{Guo} did not analyze case $(a)$.
We complement the ignored part based on Lemma \ref{2.3}.
\end{remark}

\section{Combinatorial Yamabe flow}\label{2021123}
To show Theorem \ref{123456}, let us introduce some functions motivated by Lyapunov theory \cite{CYR,BK,T3}. Recall that $\Phi(w)=-\int_{c}^{w} \sum_{i=1}^n B_{i}dw_{i}$. We set
\[
\Psi(w)=\Phi(w)+\sum_{i=1}^n b_{i}w_{i}.
\]
Because $\Phi$ is strictly convex, $\Psi$ is also a strictly convex function in $W$.
Applying Theorem \ref{1234}, we can find a unique point $w^{\ast}\in W$ such that $B_{i}(w^{\ast})=b_{i}$. It follows that
\[
\frac{\partial\Psi}{\partial w_{i}} \,\bigg|_{w=w^{\ast}}=-B_{i}(w^{\ast})+b_{i}=0.
\]
Hence $w^{\ast}$ is a critical point of $\Psi$.
We now consider the following function
\begin{equation} \nonumber
\Lambda(w)=\Psi(w)-\Psi(w^\ast)+C(w),
\end{equation}
where
\[
C(w)=\sum_{i=1}^n(B_{i}-b_{i})^2.
\]
\begin{lemma} \label{20211224}
$\Lambda(w)\to+\infty$ as $w$ approaches the boundary of $W$.
\end{lemma}
\begin{proof}
The proof is split into the following two situations:
\begin{itemize}
\item[$(i)$] There exists $i$ such that $w_{i}\to-\infty$ or $w_{i}\to+\infty$. No matter which case occurs, we have  $\|w\| \to +\infty$. Since $\Psi$ is strictly convex and $w^{\ast}$ is a critical point, the following Lemma \ref{20211223} implies
\[
\lim_{\|w\| \to +\infty} \Psi(w)=+\infty.
\]
Thus $\Lambda(w)\to+\infty$.
\item[$(ii)$] For at least one ideal face $f=ijk$, the following property holds:
\begin{equation}
\nonumber
w_{i}+w_{j}\to-\ln\cosh\frac{l_{ij}^{0}}{2}.
\end{equation}
It follows from Lemma \ref{2.4} that $B_{i}\to+\infty$. Therefore, $C(w)\to+\infty$.
Meanwhile, applying Lemma \ref{20211223}, we derive that $w^{\ast}$ is a global minimal point of $\Psi$. Namely,
\[
\Psi(w)\geq\Psi(w^{\ast}).
\]
As a result, $\Lambda(w)\to+\infty$.
\end{itemize}

To summarize, we finish the proof.
\end{proof}

\begin{lemma}\label{20211223}
Let h be a smooth strictly convex function defined in a convex set $\Omega$ with a critical point $p\in \Omega$. Then the following properties hold:
\begin{itemize}
\item[$(i)$] $p$ is the unique global minimum point of $h$.
\item[$(ii)$] If $\Omega$ is unbounded, then $\lim_{\|x\| \to +\infty} h(x)=+\infty$.
\end{itemize}
\end{lemma}

This is an elementary result. One refers to \cite{GHZ,GHZ2} for a proof. Now we are ready to prove Theorem \ref{123456}.

\medskip
\begin{proof} [Proof of Theorem \ref{123456}] First we prove the flow exists all the time. Since each $B_{i}$ depends on $w$ smoothly, $(B_{1},B_{2},\cdots,B_{n})$ is locally Lipschitz continuous. By classical ODE theory \cite{T3}, the flow (\ref{1111111111}) has a unique solution $w(t)$ on $[0,\epsilon)$ for some  $\epsilon>0$. Consequently, $w(t)$ exists in a maximal time interval $[0,T_{0})$ with $0<T_{0}\leq +\infty$. It suffices to show $T_{0}=+\infty$. Assume on the contrary that $T_{0}$ is finite. Then there exists $t_{m}\to T_{0}$ such that $w(t_{m})$ approaches the boundary of $W$.
By Lemma \ref{20211224}, we have
\[
\Lambda(w(t_{m}))\to+\infty.
 \]
However, a direct computation yields
\begin{equation}
\begin{aligned}
\nonumber
\frac{d\Lambda}{dt}
&=\sum_{i=1}^n \frac{\partial \Psi}{\partial w_{i}}\frac{dw_{i}}{dt}+2\sum_{i=1}^n\sum_{j=1}^n(B_{i}-b_{i})\frac{\partial B_{i}}{\partial w_{j}}\frac{dw_{j}}{dt}\\
&=-\sum_{i=1}^n(B_{i}-b_{i})^2+2(B_{1}-b_{1},\cdots,B_{n}-b_{n})L(B_{1}-b_{1},\cdots,B_{n}-b_{n})^{T}\\
&\leq0.
\end{aligned}
\end{equation}
It follows that
\[
\Lambda(w(t_{m}))\leq\Lambda(w(0)),
\]
which leads to a contradiction. Consequently, $T_0=+\infty$. That means $w(t)$ exists for all time. In addition, similar reasoning implies $w(t)$ stays in a compact set of $W$.

The next step is to prove $w(t)$ converges exponentially fast to $w^\ast$.
Since  $w(t)$ stays in a compact set of $W$ and $L$ depends on $w$ continuously, there exists $\lambda_{0}>0$ such that
\begin{equation}
\begin{aligned}
\label{20211027}
\nonumber
C^\prime(w)&=2\sum_{i=1}^n\sum_{j=1}^n(B_{i}-b_{i})\frac{\partial B_{i}}{\partial w_{j}}\frac{dw_{j}}{dt}\\
&=2(B_{1}-b_{1},\cdots,B_{n}-b_{n})L(B_{1}-b_{1},\cdots,B_{n}-b_{n})^{T}\\
&\leq-2\lambda_{0}\bigg[\sum_{i=1}^n(B_{i}-b_{i})^2\bigg]\\
&=-2\lambda_{0} C(w).
\end{aligned}
\end{equation}
As a result,
\[
\sum_{i=1}^n(B_{i}-b_{i})^2 =C(w)\leq C(w(0))e^{-2\lambda_{0} t}.
\]
Then
\[
|B_{i}-b_{i}|\leq \sqrt{C(w(0))}e^{-\lambda_{0}t},
\]
which yields
\[
|w_{i}-w^{\ast}_{i}|=\left|\int_{\infty}^{t}(B_{i}-b_{i})dt\right|\leq\frac{\sqrt{C(w(0))}}{\lambda_{0}}e^{-\lambda_{0}t}.
\]
We thus prove that $w$ converges exponentially fast to $w^{\ast}$.
\end{proof}

\begin{remark}
In fact, $\Lambda(w)$ is a proper Lyapunov function of the combinatorial Yamabe flow. From Barbashin-Krasovskii Theorem \cite{BK}, it follows that $w^{\ast}$ is a globally asymptotically stable equilibrium point of \eqref{1111111111}. Namely, each point $w(0)\in W$ is attracted to $w^\ast$ by the system. This provides an alternative approach to the convergence of $w(t)$.
\end{remark}

\acknowledgements{\rm The first author would like to thank Ren Guo, Qianghua Luo, Yaping Xu and Te Ba for their encouragement and comments. The second author thanks Yanwen Luo for comments and communications.}

\noindent {Shengyu Li, lishengyu@hnu.edu.cn\\[2pt]
\emph{School of Mathematics, Hunan University, Changsha 410082, P.R. China.}\\[2pt]

\noindent Xu Xu, xuxu2@whu.edu.cn\\[2pt]
\emph{School of Mathematics and Statistics, Wuhan University, Wuhan 430072, P.R.China.}\\[2pt]

\noindent Ze Zhou, zhouze@hnu.edu.cn\\[2pt]
\emph{School of Mathematics, Hunan University, Changsha, 410082, P.R. China.}

\end{document}